\documentclass[12pt]{article}
\usepackage{amsmath}
\usepackage{amsthm, amscd, amssymb, amsfonts, amsbsy}
\usepackage[usenames, dvipsnames]{color}
\usepackage{enumerate}
\usepackage[colorlinks=true,linkcolor=blue,citecolor=red]{hyperref}
\usepackage{mathrsfs}
\usepackage{verbatim}
\usepackage[utf8]{inputenc}
\usepackage{marginnote}
\usepackage{xcolor}
\usepackage{pagecolor}
\usepackage{cite}

\numberwithin{equation}{section}

\theoremstyle{plain}
\newtheorem{theorem}{Theorem}[section]
\newtheorem{lemma}[theorem]{Lemma}
\newtheorem{proposition}[theorem]{Proposition}
\newtheorem{corollary}[theorem]{Corollary}

\newtheorem{innercustomthm}{Theorem}
\newenvironment{customthm}[1]
  {\renewcommand\theinnercustomthm{#1}\innercustomthm}
  {\endinnercustomthm}

\theoremstyle{definition}
\newtheorem{definition}[theorem]{Definition}

\theoremstyle{remark}
\newtheorem{remark}[theorem]{Remark}

\makeatletter
\def\dashint{\operatorname%
{\,\,\text{\bf--}\kern-.98em\DOTSI\intop\ilimits@\!\!}}
\makeatother

\def\bR{\mathbb{R}}

\def\bT{\mathbb{T}}

\def\MR#1{}

\newlength{\defbaselineskip}
\begin{document}
\title{Regular solutions of the stationary Navier-Stokes equations on high dimensional Euclidean space}

\author{YanYan Li\footnote{Department of
Mathematics, Rutgers University, 110 Frelinghuysen Rd, Piscataway,
NJ 08854, USA. Email: yyli@math.rutgers.edu.}~\footnote{Partially supported by NSF Grants DMS-1501004, DMS-2000261, and Simons Fellows Award 677077.}\quad and \quad Zhuolun Yang\footnote{Department of Mathematics, Rutgers University, 110 Frelinghuysen Rd, Piscataway,
NJ 08854, USA. Email: zy110@math.rutgers.edu.}~\footnote{Partially supported by NSF Grants DMS-1501004 and DMS-2000261.}}
\date{}
\maketitle

\begin{abstract}
We study the existence of regular solutions of the incompressible stationary Navier-Stokes equations in $n$-dimensional Euclidean space with a given bounded external force of compact support. In dimensions $n\le 5$, the existence of such solutions was known. In this paper, we extend it to dimensions $n\le 15$.
\end{abstract}

\section{Introduction and main results}
The incompressible stationary Navier-Stokes equations that describe the motion of a steady-state viscous fluid are formulated as follows (with viscosity $\nu = 1$):

\begin{equation}\label{NS_Omega}
\left\{
\begin{aligned}
- \Delta u + (u \cdot \nabla) u + \nabla p &= f,\\
\mbox{div }u &= 0,
\end{aligned}
\right.
\end{equation}
where $u$ and $f$ are vector fields representing velocity and external force respectively, $p$ is a scalar function representing pressure. Let $f$ be a bounded external force, we say that $(u,p)$ is a regular solution of \eqref{NS_Omega}, if
$$u \in W^{2,s}_{loc}, \quad \mbox{and} \quad p \in W^{1,s}_{loc},$$
for any $s < \infty$.

We are interested in the existence of regular solutions of \eqref{NS_Omega}, in dimensions $n \ge 5$. Such existence results are classical in dimensions $n =2,3$, see, e.g., \cite{T}, while in dimension $n = 4$ it follows from Gerhardt \cite{Ger}. The problem \eqref{NS_Omega} is classified as "super-critical" in dimensions $n \ge 5$. Frehse and Rů\v{z}i\v{c}ka \cite{FR0} showed that in a bounded domain in $\bR^5$ with Dirichlet boundary data $u = 0$, problem \eqref{NS_Omega} has certain weak solutions which are “almost regular”. Struwe \cite{S} established on $\bR^5$ and on torus $\bT^5$ a $C^1$ a-priori bound of solutions and proved the existence of regular solutions. Frehse and Rů\v{z}i\v{c}ka established in \cite{FR3,FR4} the a-priori bound of solutions and the existence of regular solutions in $\bT^n$ for $5 \le n \le 15$, and produced in \cite{FR1, FR2} weak solutions of the Dirichlet problem that are regular in the interior in dimension $n=5,6$. We refer to \cite{T2, Ko} and \cite[Chapter 7]{BF} for simplified proofs and more discussions on this subject.

For small data, the existence of regular solutions of the Dirichlet problem in any dimension were studied by Farwig and Sohr \cite{FS}.

In this paper, we consider the stationary Navier-Stokes equations on the Euclidean space:
\begin{equation}\label{NS}
\left\{
\begin{aligned}
- \Delta u + (u \cdot \nabla) u + \nabla p &= f\\
\mbox{div }u &= 0
\end{aligned}
\right.
\quad \mbox{in } \bR^n,
\end{equation}
and extend the above mentioned result in \cite{S} for $n = 5$ to $n \le 15$. Our main result is as follows.

\begin{theorem}\label{main}
For $5 \le n \le 15$ and $f \in L^\infty(\bR^n)$ with compact support, there exists a regular solution $(u,p)$ of \eqref{NS}. Furthermore, the solution $(u,p)$ satisfies
\begin{align}\label{reasonable_decay}
|u(x)| \le \frac{C}{(1+ |x|)^{n - 2}}, \quad |\nabla u(x)|+ |p(x)| \le \frac{C}{(1+ |x|)^{n - 1}}, \quad \forall x \in \bR^n,
\end{align}
where $C > 0$ depends only on $n$, an upper bound of the diameter of supp$(f)$, and an upper bound of $\|f\|_{L^\infty(\bR^n)}$.
\end{theorem}

\begin{remark}
For simplicity, the external force is assumed to be compactly supported. Indeed, our proof works for $f$ with sufficient decay at infinity.
\end{remark}

\begin{remark}
If in addition, $f \in W^{m, \infty}$, for $m \ge 0$, then $u \in W^{m+2, s}_{loc}$, $p \in W^{m+1, s}_{loc}$ for any $s < \infty$, and for all $1 \le l \le m+1$,
$$|\nabla^l u(x)| + |\nabla^{l-1} p(x)| \le \frac{C}{(1+ |x|)^{n - 2 + l}}, \quad \forall x \in \bR^n.$$
where $C > 0$ depends only on $n$, an upper bound of the diameter of supp$(f)$, and an upper bound of $\|f\|_{W^{m,\infty}(\bR^n)}$. This follows from standard estimates for stationary Stokes equations.
\end{remark}

One related question is whether $H^1$ weak solutions of \eqref{NS_Omega} are regular. An affirmative answer is classical in dimensions $n=2$, $3$. The case $n = 4$ was proved by Gerhardt \cite{Ger}. Giaquinta and Modica \cite{GM} proved that $H^1$ weak solutions are regular for a class of nonlinear systems including the stationary Navier-Stokes in dimensions $n \le 4$. The question remains open in dimensions $n \ge 5$. Sohr \cite{So} showed that $u \in H^{1} \cap L^n$ is regular in any dimension. As a consequence of the techniques developed in \cite{FR0,FR1,FR2,FR3,FR4}, Frehse and Rů\v{z}i\v{c}ka gave in \cite{FR5} a new regularity criterion which improved the result of \cite{So}. 

Starting from the groundbreaking work of De Lellis and Sz\'{e}kelyhidi Jr. \cite{DeLS1}, there has been much development in applications of the convex integration method in connection with the Euler and Navier-Stokes equations; see the survey papers \cite{DeLS, BV}. Buckmaster and Vicol \cite{BV1} recently proved the nonuniqueness of weak solutions of the 3D evolutionary Navier-Stokes equations with finite energy using the convex integration method; see also \cite{BCV, Luo, CL} for related works. In particular, it was shown in \cite{Luo} that there exists a non-regular solution $u$ of \eqref{NS_Omega} on $\bT^n$ with $n \ge 4, f = 0$, which lies in $H^\beta(\bT^n)$ for any $\beta < \frac{1}{200}$. It was also pointed out that the regularity can be improved to $H^\beta$ for any $\beta < \frac{1}{2}$ when $n$ is sufficiently large.

In a seminal paper \cite{CKN}, Caffarelli, Kohn and Nirenberg proved that the 1-dimensional Hausdorff measure of the singular set of a suitable weak solution to the 3D evolutionary Navier-Stokes equations is zero. Partial regularity results for stationary Navier-Stokes equations were established by Struwe \cite{S2} in dimension $n = 5$, and by Dong and Strain \cite{DS} in dimension $n = 6$. For suitable weak solutions, they established an $\varepsilon$-regularity criterion in terms of a scaling invariant quantity of $\nabla u$. This implies that $u$ is regular outside a set of zero $n-4$ dimensional Hausdorff measure. The above results were extended up to the boundary by Kang \cite{K} and Dong and Gu \cite{DG} in dimensions $n = 5,6$ respectively. Tian and Xin \cite{TX} established an $\varepsilon$-regularity criterion in terms of a scaling invariant quantity of the vorticity for smooth solutions in any dimension.

For some other related studies on incompressible stationary Navier-Stokes equations on the Euclidean space, see, e.g., \cite{ST,JS, KS} and the references therein.

Theorem \ref{main} is proved by establishing a-priori estimates of solutions in appropriate function spaces, and then applying the Leray-Schauder degree theory. Our proof is based on the results and methods developed in the work of Frehse and Rů\v{z}i\v{c}ka \cite{FR0, FR1, FR2, FR3, FR4, FR5, FR6}, Struwe \cite{S}, and Tian and Xin \cite{TX}.

Let
\begin{equation}\label{Stokes_fundamental_UP}
\left\{
\begin{aligned}
U_{ij}(x) &= \frac{1}{2n\omega_n} \left[ \frac{\delta_{ij}}{(n-2)|x|^{n-2}} + \frac{x_i x_j}{|x|^n} \right],\\
P_j(x) &= \frac{1}{n\omega_n} \frac{x_j}{|x|^n},
\end{aligned}
\right.
\end{equation}
denote the fundamental solution of the stationary Stokes equations. That is, for each fixed $j$, we have
\begin{equation*}\label{Stokes_fundamental}
\left\{
\begin{aligned}
- \Delta U_{ij} + \partial_i P_j &= \delta_{ij}\delta_0,\\
\partial_i U_{ij} &= 0,
\end{aligned}
\right.
\end{equation*}
where $\delta_{ij}$ is the Kronecker delta ($\delta_{ij} = 0$ for $i \neq j$ and $\delta_{ii} = 1$) 
and $\delta_0$ is the Dirac mass at the origin.

Instead of working with \eqref{NS}, we will work with the following integral equation 
\begin{equation}\label{integral_equation}
u_i(x) = \int_{\bR^n} U_{ij}(x-y) \left(f_j(y) - u_k(y) \partial_k u_j(y)\right) \, dy,
\end{equation}
and find a solution $u$ with proper decay at infinity.
We define the space $C_d^1(\bR^n)$ to be the closure of $C_{c,\sigma}^\infty(\bR^n)$, the space of smooth, divergence-free, and compactly supported vector fields on $\bR^n$, under the norm
$$\|u\|_{C_d^1(\bR^n)}:= \left\| (1+ |\cdot|)^{n - 3} u \right\|_{L^\infty(\bR^n)} + \left\| (1+ |\cdot|)^{n - 2} \nabla u \right\|_{L^\infty(\bR^n)}.$$

To prove Theorem \ref{main}, we only need to show the existence of a solution $u \in C_d^1(\bR^n)$ of \eqref{integral_equation}. For such a solution $u$, let
\begin{equation}\label{pressure}
p(x) := \int_{\bR^n} P_j(x - y) \left(f_j(y) - u_k(y) \partial_k u_j(y)\right) \, dy.
\end{equation}
One can verify that $(u,p)$ solves the stationary Navier-Stokes equation \eqref{NS}.
\begin{remark}
In the definition of the space $C_d^1(\bR^n)$, if the exponents $n-3$ and $n-2$ are replaced by $\alpha$ and $\alpha+1$, for any $\alpha \in (1,n-2)$, our proof will go through essentially the same way. The solution $u$ we find enjoys a better decay as stated in \eqref{reasonable_decay}. The reason we choose an exponent $\alpha < n-2$ is to ensure that the operator defined by the right hand side of \eqref{integral_equation} is a compact operator from $C_d^1(\bR^n)$ to itself.
\end{remark}

The following crucial a-priori estimate allows us to show the existence of a solution of \eqref{integral_equation} in $C_d^1(\bR^n)$ using the Leray-Schauder degree theory.

\begin{theorem}\label{decay_u_prop}
For $5 \le n \le 15$, and for $f \in L^\infty(\bR^n)$ with compact support, let $u \in C_d^1(\bR^n)$ be a solution of \eqref{integral_equation}. Then
\begin{equation*}\label{u_decay_norm_control}
\|u\|_{C_d^1(\bR^n)} \le C,
\end{equation*}
where $C > 0$ depends only on $n$, an upper bound of the diameter of supp$(f)$, and an upper bound of $\|f\|_{L^\infty}$.
\end{theorem}

The remaining part of this paper is organized as follows. Some preliminary a-priori estimates are proved in Section 2. Theorem \ref{decay_u_prop} is proved in Section 3. In Section 4, we apply the Leray-Schauder degree theory to complete the proof of Theorem \ref{main}, using Theorem \ref{decay_u_prop}.\\

\section{Some preliminary a-priori estimates}

In this section, we give some preliminary estimates on solutions $u$ of \eqref{integral_equation} in $C_d^1(\bR^n)$. First, we present a calculus lemma that can be verified easily.

\begin{lemma}\label{calculus_lemma}
Let
$$F(x) := \int_{\bR^n} |x - y|^{- \alpha} (1 + |y|)^{- \beta} \, dy, \quad x \in \bR^n,$$
with $0 \le \alpha < n, \alpha + \beta > n$. Then for all $x \in \bR^n$,
$$|F(x)| \le \left\{
\begin{aligned} 
&C \log (2 + |x|) (1 + |x|)^{-\alpha},  &&\mbox{if}~\beta = n,\\
&C (1 + |x|)^{-\gamma}, &&\mbox{if}~\beta \neq n,
\end{aligned}
\right.$$
where $\gamma = \min\{ \alpha, \alpha + \beta - n \}$, $C>0$ depends only on $\alpha, \beta$ and $n$.\\
\end{lemma}

\begin{proof}
In the following, $C$ denotes some positive constants depending
only on $\alpha, \beta$ and $n$ whose values may change from line to line.

Since $\alpha<n$ and $\alpha+\beta>n$, the inequality is clear for $|x|\le 2$.
So we will assume
  $|x|>2$.  Define,
$$
\Omega_1:=\{y : |y-x|\le |x|/2\};
\  
\Omega_2:= \{y :  |x|/2 \le |y-x|\le 4|x|\};
\  
\Omega_3:=\{ y : 
|y-x|\ge 4|x|\}.
$$
It is easy to see that
$$
|x|/2\le |y|\le 3|x|/2, \ \forall\ y\in \Omega_1;
\ \  |y|\le 5|x|\ \forall\ y\in \Omega_2; \  \
|y-x|\ge \frac 45 |y|\ge \frac {12}5 |x|,\ \forall y\in \Omega_3.
$$
It follows that
\begin{align*}
F(x)&\le  
C
\sum_{ i=1}^3 \int_{ \Omega_i} |x-y|^{-\alpha}(1+|y|)^{-\beta}dy\\ 
&\le  C\left\{
|x|^{-\beta} \int_{ \Omega_1 } |x-y|^{-\alpha} dy
+
|x|^{-\alpha}  \int_{ \Omega_2 } (1+|y|)^{-\beta}dy
+  \int_{ \Omega_3 } |y|^{-\alpha-\beta}dy\right\}\\
&=: I+II+III.
\end{align*}
Since
$$
I+III\le C |x|^{ n-\alpha-\beta },
$$
and
$$
II \le \left\{
\begin{aligned}
&C |x|^{ n-\alpha-\beta }, && \beta < n;\\
&C (\log |x|)|x|^{-\alpha}, && \beta = n;\\
&C |x|^{-\alpha}, && \beta > n.
\end{aligned}
\right.
$$
Lemma \ref{calculus_lemma} is proved.\\
\end{proof}

For $u \in C_d^1(\bR^n)$, since
\begin{equation}\label{u_decay}
|u| \le \|u\|_{C_d^1}(1+|x|)^{-(n-3)} \quad \mbox{and} \quad |\nabla u| \le \|u\|_{C_d^1}(1+|x|)^{-(n-2)},
\end{equation}
the corresponding pressure given by \eqref{pressure} satisfies
$$|p(x)| \le C\left( \|f\|_{C_d^1} + \|u\|^2_{C_d^1} \right) \int_{\bR^n} |x - y|^{- (n-1)} (1 + |y|)^{- (2n-5)} \, dy,$$
where $C > 0$ depends only on $n$. Therefore, by Lemma \ref{calculus_lemma}, we have
\begin{equation}\label{p_decay}
|p(x)| \le \left\{
\begin{aligned}
&C\left( \|f\|_{C_d^1} + \|u\|^2_{C_d^1} \right) \log (2 + |x|) (1 + |x|)^{-(n-1)},  &&\mbox{when}~n = 5,\\
&C\left( \|f\|_{C_d^1} + \|u\|^2_{C_d^1} \right) (1 + |x|)^{-(n-1)}, &&\mbox{when}~n \ge 6,\\
\end{aligned}
\right.
\end{equation}
where $C > 0$ depends only on $n$.
\\

Now we give an initial a-priori estimate for $(u,p)$.

\begin{lemma}\label{base}
For $n \ge 5$, and for $f \in L^\infty(\bR^n)$ with compact support, let $u \in C_d^1(\bR^n)$ be a solution of \eqref{integral_equation} and $p$ be given by \eqref{pressure}, then
\begin{align}
\|\nabla u\|_{L^2(\bR^n)} + \| u \|_{L^{\frac{2n}{n-2}}(\bR^n)} &\le C\|f\|_{L^{\frac{2n}{n+2}}(\bR^n)}, \label{u_base}\\
\| \nabla p\|_{L^{\frac{n}{n-1}}(\bR^n)} + \| p\|_{L^{\frac{n}{n-2}}(\bR^n)} &\le C\left(\|f\|_{L^{\frac{2n}{n+2}}(\bR^n)}^2 + \|f\|_{L^{\frac{n}{n-1}}(\bR^n)}\right), \label{p_base}
\end{align}
where $C > 0$ depends only on $n$.
\end{lemma}

\begin{proof}
As mentioned before, $(u,p)$ solves the stationary Navier-Stokes equations \eqref{NS}. The proof follows from a standard energy estimate argument. Let $\eta$ be a smooth cut-off function such that $\eta \equiv 1$ in $B_R$, $\eta \equiv 0$ outside $B_{2R}$, and $|\nabla \eta| \le CR^{-1}$. In the following, $C'$ denotes a constant which is allowed to depend on $u$, but is independent of $R$. We multiply \eqref{NS} by $u \eta^2$, and integrate by parts,
\begin{align*}
\int |\nabla (u \eta)|^2 - \int |u|^2|\nabla \eta|^2 - \int |u|^2 u \cdot \nabla \eta \eta - 2\int p u \cdot \nabla \eta \eta = \int f\cdot u \eta^2.
\end{align*}
By \eqref{u_decay}, \eqref{p_decay}, and H\"older's inequality, we have
\begin{align*}
\int_{B_R} |\nabla u|^2 &\le \frac{C}{R^2} \int_{B_{2R} \setminus B_R} |u|^2 + \frac{C}{R} \int_{B_{2R} \setminus B_R} |u|^3 + \frac{C}{R} \int_{B_{2R} \setminus B_R} |p| |u| + \int_{B_{2R}} |f| |u|\\
&\le \frac{C'}{R^{n-4}} + \frac{C'}{R^{2n-8}} + \frac{C' \log R}{R^{n-3}} +  \|f\|_{L^{\frac{2n}{n+2}}(B_{2R})} \|u\|_{L^{\frac{2n}{n-2}}(B_{2R})},
\end{align*}
when $R$ is large. Taking $R \to \infty$ will yield
$$\| \nabla u \|_{L^2(\bR^n)}^2 \le \|f\|_{L^{\frac{2n}{n+2}}(\bR^n)} \|u\|_{L^{\frac{2n}{n-2}}(\bR^n)} \le C\|f\|_{L^{\frac{2n}{n+2}}(\bR^n)} \| \nabla u\|_{L^2(\bR^n)}$$
by Poincare inequality, which implies
$$\|u\|_{L^{\frac{2n}{n-2}}(\bR^n)} + \| \nabla u\|_{L^2(\bR^n)} \le C\|f\|_{L^{\frac{2n}{n+2}}(\bR^n)}.$$
Then we have
$$\| (u \cdot \nabla) u \|_{L^{\frac{n}{n-1}}(\bR^n)} \le \|u\|_{L^{\frac{2n}{n-2}}(\bR^n)} \|\nabla u\|_{L^2(\bR^n)} \le C\|f\|^2_{L^{\frac{2n}{n+2}}(\bR^n)},$$
and \eqref{p_base} follows from potential estimates applying on the representation \eqref{pressure}.\\
\end{proof}

We denote the total head pressure by
$$\theta:= \frac{|u|^2}{2} + p.$$
This quantity has played an important role in the study of the stationary Navier-Stokes equations. It was already observed by Gilbarg and Weinberger \cite{GW} that it satisfies an elliptic equation
\begin{equation}\label{headpressure}
-\Delta \theta + u \cdot \nabla \theta = - |\partial_i u_j - \partial_j u_i|^2 + f\cdot u - \mbox{div }f.
\end{equation}
The following a-priori estimate on $u$ in terms of the $L^r$ norm of $\theta_+ :=\max\{\theta,0\}$, $r > n/2$,  can be deduced from the work \cite{FR0,FR1,TX}.

\begin{proposition}\label{aprioriestimate}
For $n \ge 2$, $r > n/2$, and $f \in L^\infty(B_1)$, let $(u,p)$ be a regular solution of the stationary Navier-Stokes equations
\begin{equation}\label{NS_B1}
\left\{
\begin{aligned}
- \Delta u + (u \cdot \nabla) u + \nabla p &= f\\
\mbox{div }u &= 0
\end{aligned}
\right.
\quad \mbox{in } B_1:= \{x \in \bR^n~\big|~|x|<1\}.
\end{equation}
Assume
$$\|u\|_{W^{1,2}(B_1)} + \|p\|_{W^{1,n/(n-1)}(B_1)} + \|f\|_{L^\infty(B_1)}+ \|\theta_+\|_{L^r(B_1)} \le C_0$$
for some constant $C_0$, then
\begin{equation}\label{apriori}
\|u\|_{L^\infty(B_{1/2})} + \|\nabla u\|_{L^\infty(B_{1/2})} \le C,
\end{equation}
where $C > 0$ depends on $n$, $C_0$, and a positive lower bound of $r - n/2$.
\end{proposition}

\begin{remark}
We will give in this section a short proof of Proposition \ref{aprioriestimate} using results from \cite{FR1} and \cite{TX}. Proposition \ref{aprioriestimate} can also be deduced through arguments in \cite{FR1} and \cite{FR6}, which will be presented in the Appendix.
\end{remark}

Before proving Proposition \ref{aprioriestimate}, let us recall the previously mentioned $\varepsilon$-regularity criterion by Tian and Xin:

\begin{customthm}{A}[\hspace{-0.01in}\cite{TX}]\label{regularity_criterion}
For $n \ge 2$ and $f \in L^\infty(B_1)$, let $(u,p)$ be a regular solution of \eqref{NS_B1}, with
$$\|u\|_{L^2(B_1)} \le M_0,$$
for some constant $M_0$.
There is a positive constant $\varepsilon_0$ depending only on $n$ and $M_0$, such that if for some $R_0 > 0$,
$$r^{-(n-4)}\int_{B_r(x_0)}|\partial_i u_j - \partial_j u_i|^2 < \varepsilon_0, \quad \mbox{for all } 0 < r < R_0,~ x_0 \in B_{1/2},$$
then there exists a positive constant $R_1$ depending only on $n, R_0, M_0$, and an upper bound of $\|f\|_{L^\infty(B_1)}$, such that
$$\sup_{B_{r/2}}|\nabla u| \le Cr^{-2}, \quad \mbox{for all } 0 < r < R_1,$$
where $C$ is a positive constant depending only on $n$, $M_0$ and an upper bound of $\|f\|_{L^\infty(B_1)}$.\\
\end{customthm}

\begin{remark}
The corresponding theorem stated in \cite{TX} is for $f = 0$. However, their proof can be modified to allow nonzero $f$. Indeed, we can replace their representation formula (2.20) on \cite[Page 227]{TX} by
$$w(x) = (\nabla \Gamma \ast ((n-1)w \wedge u + ^{\ast}f))(x) + H_1(x), \quad x \in B_{1/2},$$
with the convolution integral over $B_1$. Here $\Gamma$ is the fundamental solution of the Laplace equation, $w(x) = ^{\ast}du(x)$, $^{\ast}$ denotes the Hodge star operator, and $H_1$ is harmonic in $B_1$. This representation formula can be found in Section 4 of \cite{Ser}.\\
\end{remark}

\begin{proof}[Proof of Proposition \ref{aprioriestimate}]
Following the arguments in the proof of Theorem 1.5 in \cite{FR1}, we have,
\begin{equation}\label{grad_u_weighted}
\int_{B_R(x_0)} \frac{|\nabla u|^2}{|x - x_0|^{n-4}} \le CR^\beta  \quad \mbox{for any}~x_0 \in B_{1/2}, 0<R <1/4,
\end{equation}
where $C$ and $\beta$ are positive constants depending only on $n$, $C_0$, and a positive lower bound of $r - n/2$; see the last line of page 372 and the first two lines of page 373 for the statement, as well as Lemma 3.5, Lemma 3.1 and (2.14) in the paper.

It follows from \eqref{grad_u_weighted} that
\begin{equation}\label{u_morrey}
\frac{1}{R^{n-4}} \int_{B_R(x_0)} |\nabla u|^2 \le C_1R^\beta \quad \mbox{for any}~x_0 \in B_{1/2}, 0<R <1/4.
\end{equation}
We choose $\varepsilon_0$ as in Theorem \ref{regularity_criterion} with $M_0 = \|u\|_{L^2(B_1)}$, and choose $R_0$ satisfying $C_1 R_0^\beta < \varepsilon_0$. Then
$$\frac{1}{R^{n-4}} \int_{B_R(x_0)} |\nabla u|^2 \le \varepsilon_0 \quad \mbox{for any}~x_0 \in B_{1/4}, 0<R <R_0.$$
By Theorem \ref{regularity_criterion}, we have
$$|\nabla u(0)| \le C,$$
where $C>0$ depends on $n, r,$ and $C_0$. Since the problem is translation invariant, we actually have
$$\|\nabla u\|_{L^\infty(B_{1/2})} \le C.$$
Boundedness of $u$ follows from the interpolation inequality:
$$\|u\|_{L^\infty(B_{1/2})} \le C \left( \|u\|_{L^2(B_{1/2})} + \|\nabla u\|_{L^\infty(B_{1/2})} \right) \le C.$$
\end{proof}

Next, we prove the following proposition:

\begin{proposition}\label{theta}
For $5 \le n \le 15$, and $f \in L^\infty(\bR^n)$ with compact support, let $u \in C_d^1(\bR^n)$ be a solution of \eqref{integral_equation} and $p$ be given by \eqref{pressure}. Then there exists an $r > n/2$ depending only on $n$, such that
\begin{align*}
 \| \theta_+ \|_{L^r (\bR^n)} &\le C(r,f),
\end{align*}
where $C(r,f) > 0$ depends only on $n$, $r$, an upper bound of the diameter of supp$(f)$, and an upper bound of $\|f\|_{L^\infty}$.
\end{proposition}

The following a-priori estimate is a consequence of Proposition \ref{aprioriestimate}, Lemma \ref{base} and Proposition \ref{theta}.

\begin{corollary}\label{u_C1_coro}
For $5 \le n \le 15$, and $f \in L^\infty(\bR^n)$ with compact support, let $u \in C_d^1(\bR^n)$ be a solution of \eqref{integral_equation}. Then
\begin{equation}\label{u_C1}
\| u \|_{L^\infty(\bR^n)} + \| \nabla u \|_{L^\infty (\bR^n)} \le C(f),
\end{equation}
where $C(f) > 0$ depends only on $n$, an upper bound of the diameter of supp$(f)$, and an upper bound of $\|f\|_{L^\infty}$.\\
\end{corollary}

We will prove Proposition \ref{theta} through the following lemmas.

\begin{lemma}\label{theta_u_relation_1_lemma}
For $f \in L^\infty(\bR^n)$ with compact support, let $u \in C_d^1(\bR^n)$ be a solution of \eqref{integral_equation} and $p$ be given by \eqref{pressure}. Then for any $\frac{2n}{n+2} \le q < \frac{n}{2}$, we have
\begin{equation}\label{theta_u_relation_1}
\| \theta_+ \|_{L^{\frac{nq}{n-2q}}(\bR^n)} \le C(q,f) \left( \|u\|_{L^q(\{supp(f)\})} + 1 \right),
\end{equation}
where $C(q,f) > 0$ depends only on $n, q,$ an upper bound of the diameter of supp$(f)$, and an upper bound of $\|f\|_{L^\infty}$.
\end{lemma}

\begin{remark}
On $\bT^n$, the estimate
$$\| \theta_+ \|_{L^{\frac{nq}{n-2q}}(\bT^n)} \le C \left( \|u\|_{L^q(\bT^n)} + 1 \right)$$
was proved in \cite[Proposition 3.3]{FR4}.
\end{remark}

\begin{proof}
Let $\eta$ be a smooth cut-off function such that $\eta \equiv 1$ in $B_R$, $\eta \equiv 0$ outside $B_{2R}$, and $|\nabla \eta| \le CR^{-1}$. In the following, we use $C'$ to denote a constant which is allowed to depend on $u$, but is independent of $R$.
Note that by Lemma \ref{base}, \eqref{u_decay}, \eqref{p_decay}, and the definition of $\theta$, we have
\begin{equation}\label{theta_properties}
|\theta| \le \frac{C'\log (2 + |x|)}{(1+|x|)^{n-1}} \quad \mbox{and} \quad \| \nabla \theta\|_{L^{\frac{n}{n-1}}} \le C\left(\|f\|_{L^{\frac{2n}{n+2}}}^2 + \|f\|_{L^{\frac{n}{n-1}}}\right),
\end{equation}
where $C$ is a positive constant depending only on $n$.  Take 
\begin{equation}\label{s_formula}
s = \frac{nq-n}{n-2q},
\end{equation}
we multiply \eqref{headpressure} by $\theta_+^s \eta^2$ and integrate by parts, we have
\begin{align*}
&s\int_{\bR^n} |\nabla \theta_+|^2 \theta_+^{s-1} \eta^2 + 2\int_{\bR^n}\nabla \theta_+ \cdot \nabla \eta \theta_+^s \eta - \frac{2}{s+1}\int_{\bR^n} u \cdot \nabla \eta \theta_+^{s+1} \eta\\
&= -\int_{\bR^n} |\partial_i u_j - \partial_j u_i|^2 \theta_+^s \eta^2 + \int_{\bR^n} f\cdot u \theta_+^s \eta^2 + s\int_{\bR^n} f \cdot \nabla \theta_+ \theta_+^{s-1} \eta^2 + 2\int_{\bR^n} f\cdot \nabla \eta \theta_+^s \eta. 
\end{align*}
Note that $s \ge 1$ due to the range of $q$. We drop the first term on the right hand side and take absolute value inside the integrals, we have
\begin{align}\label{step1}
s\int_{B_R} |\nabla \theta_+|^2 \theta_+^{s-1} &\le \frac{C}{R}\int_{B_{2R} \setminus B_R}|\nabla \theta_+|\theta_+^s + \frac{C}{R}\int_{B_{2R} \setminus B_R} |u| \theta_+^{s+1}\nonumber \\
&+ C \int_{\{supp(f)\}} |u| \theta_+^s + C\int_{\{supp(f)\}} |\nabla \theta_+|\theta_+^{s-1}+ \frac{C}{R}\int_{B_{2R} \setminus B_R}\theta_+^s \nonumber \\
&=: I + II + III + IV + V,
\end{align}
where $C > 0$ depends only on $n$ and $\|f\|_{L^\infty}$. We denote $o(1)$ to be a quantity that goes to $0$ as $R \to \infty$.

When $R$ is large, for $I$, we have, by H\"older's inequality and \eqref{theta_properties}
\begin{align*}
I &\le \frac{C}{R} \left( \int_{\bR^n} |\nabla \theta_+|^{\frac{n}{n-1}} \right)^{\frac{n-1}{n}}\left( \int_{B_{2R} \setminus B_R} \theta_+^{sn}  \right)^{\frac{1}{n}}\\
 &\le \frac{C'}{R} \left( R^n \left(\frac{\log R}{R^{n-1}}\right)^{sn} \right)^{\frac{1}{n}} = o(1).
\end{align*}
We use \eqref{u_decay} and \eqref{theta_properties} to estimate
\begin{align*}
II \le \frac{C'}{R} R^n R^{-n+3} \left(\frac{\log R}{R^{n-1}}\right)^{s+1} =o(1).
\end{align*}
For $III$, we apply H\"older's inequality and get
\begin{align*}
III \le C \|u\|_{L^{q}(\{supp(f)\})} \|\theta_+^s\|_{L^{\frac{q}{q-1}}(\bR^n)}.
\end{align*}
For $IV$, we use Young's inequality and get
$$IV \le \frac{s}{2} \int_{\{supp(f)\}}  |\nabla \theta_+|^2 \theta_+^{s-1} + C \int_{\{supp(f)\}}\theta_+^{s-1}.$$
For $V$, we have, by \eqref{theta_properties},
$$V \le C' R^{n-1} \left(\frac{\log R}{R^{n-1}}\right)^{s} =o(1).$$
Note that, by \eqref{theta_properties},
$$\left| \theta_+^{\frac{sq}{q-1}} \right| \le \frac{C'[\log (2 + |x|)]^{\frac{sq}{q-1}}}{(1+|x|)^{\frac{sq(n-1)}{q-1}}},$$
and $\frac{sq(n-1)}{q-1} > n$ because of \eqref{s_formula}. Sending $R \to \infty$ in \eqref{step1}, we have
$$\int_{\bR^n} |\nabla \theta_+|^2 \theta_+^{s-1} \le C\|u\|_{L^{q}(\{supp(f)\})} \|\theta_+^s\|_{L^{\frac{q}{q-1}}(\bR^n)} + C \int_{\{supp(f)\}}\theta_+^{s-1} < \infty,$$
where $C > 0$ depends only on $n$, supp$(f)$, and $\|f\|_{L^\infty}$.
Applying Poincare inequality on the left hand side will yield
$$\left( \int_{\bR^n} \theta_+^{(s+1) \frac{n}{n-2}} \right)^{\frac{n-2}{n}} \le C \int_{\bR^n} \left|\nabla \left( \theta_+^{\frac{s+1}{2}} \right)\right|^2 = C\int_{\bR^n} |\nabla \theta_+|^2 \theta_+^{s-1},$$
where we have used \eqref{theta_properties} and \eqref{s_formula} again to justify the validity of the Poincare inequality.
Therefore
\begin{align*}
\left( \int_{\bR^n} \theta_+^{(s+1) \frac{n}{n-2}} \right)^{\frac{n-2}{n}} \le C \|u\|_{L^{q}(\{supp(f)\})} \|\theta_+^s\|_{L^{\frac{q}{q-1}}(\bR^n)} + C \int_{\{supp(f)\}}\theta_+^{s-1}.
\end{align*}
Applying H\"older's inequality and Young's inequality to the last term, we have
\begin{align*}
\int_{\{supp(f)\}}\theta_+^{s-1} \le C \left( \int_{\{supp(f)\}} \theta_+^{(s+1) \frac{n}{n-2}}  \right)^{\frac{(s-1)(n-2)}{(s+1)n}} \le \frac{1}{2}\left( \int_{\bR^n} \theta_+^{(s+1) \frac{n}{n-2}} \right)^{\frac{n-2}{n}} + C.
\end{align*}
Therefore,
\begin{align}\label{theta_1}
\left( \int_{\bR^n} \theta_+^{(s+1) \frac{n}{n-2}} \right)^{\frac{n-2}{n}} \le C \|u\|_{L^{q}(\{supp(f)\})} \|\theta_+^s\|_{L^{\frac{q}{q-1}}(\bR^n)} + C.
\end{align}
By \eqref{s_formula}, we have $(s+1) \frac{n}{n-2} = \frac{sq}{q-1} = \frac{nq}{n-2q}$.
Then estimate \eqref{theta_u_relation_1} follows from plugging \eqref{s_formula} into \eqref{theta_1} and applying Young's inequality.\\
\end{proof}

In view of Proposition \ref{aprioriestimate}, we would like to restrict $q$, such that
$$\frac{nq}{n-2q} > \frac{n}{2},$$
which is equivalent to $q > n/4$. 

The following two lemmas were obtained in \cite{FR4}, see Corollary and Theorem 4.1 on page 137 there. We provide proofs for reader's convenience.

\begin{lemma}\label{p_weighted_lemma}
For $f \in L^\infty(\bR^n)$  with $supp(f) \subset B_{R_0}$, let $u \in C_d^1(\bR^n)$ be a solution of \eqref{integral_equation} and $p$ be given by \eqref{pressure}. Then for any $0 < s < n-4$, $n/4 < q < n/2$, we have
\begin{equation*}
\int_{B_{R_0+2}} \frac{|p|}{|x-y|^{s+2}}  \, dx \le C(q,s,R_0,f) \left( \| \theta_+ \|_{L^{\frac{nq}{n-2q}}(B_{R_0 + 3})} + 1 \right), \quad \forall y \in B_{R_0+1}
\end{equation*}
where $C(q,s,R_0,f) > 0$ depends only on $n$, $q$, $s$, $R_0$, and an upper bound of $\|f\|_{L^\infty}$, and in particular, does not depend on $y$.
\end{lemma}
\begin{proof}
Fix $y \in B_{R_0+1}$. We define a smooth cut-off function $\eta_1$ satisfying
$$\eta_1 = 1 ~ \mbox{in }B_{R_0+2}; \quad \eta_1 = 0 ~ \mbox{in }B_{R_0+3}^c; \quad |\nabla \eta_1| + |\nabla^2 \eta_1| \le C.
$$
 We multiply the pressure equation 
\begin{equation}\label{pressure_equation}
- \Delta p = \partial_i u^j \partial_j u^i - \mbox{div }f
\end{equation} 
by $\eta_1|x-y|^{-s}$, where \eqref{pressure_equation} is obtained by taking divergence on \eqref{NS_B1}, and integrate by parts, after some arrangements we have
\begin{equation}\label{p_weighted_calculation_part_1}
\begin{aligned}
s&(s+2)\int_{\bR^n} \frac{|u \cdot (x-y)|^2\eta_1}{|x-y|^{s+4}}  - s\int_{\bR^n} \frac{|u |^2\eta_1}{|x-y|^{s+2}}  + s(s+2-n)\int_{\bR^n} \frac{p\eta_1}{|x-y|^{s+2}} \\
&= - \int_{\bR^n} p \Delta \eta_1 |x-y|^{-s} - 2\int_{\bR^n} p \nabla\eta_1\cdot \nabla(|x-y|^{-s}) - \int_{\bR^n} f \cdot \nabla(\eta_1 |x-y|^{-s})\\
&- 2\int_{\bR^n} u^ju^i \partial_i(|x-y|^{-s})\partial_j\eta_1 - \int_{\bR^n} u^ju^i |x-y|^{-s}\partial_{ij} \eta_1 =: RHS
\end{aligned}
\end{equation}
Since $y \in B_{R_0+1}$, $\eta_1 \equiv 0$ in $B_{R_0+3}^c$, and $supp(\nabla \eta_1) \subset B_{R_0+3} \setminus B_{R_0+2},$ we have
\begin{align*}
|RHS| &\le C \left( \int_{\{R_0+2 \le |x| \le R_0+3 \}} |p| \, dx + \int_{\{R_0+2 \le |x| \le R_0+3 \}} |u|^2\, dx + \|f\|_{L^{\infty}(\bR^n)} \right)\\
&\le C(s,f) \left( \|p\|_{L^{\frac{n}{n-2}}(\bR^n)} + \|u\|_{L^{\frac{2n}{n-2}}(\bR^n)} +\|f\|_{L^{\infty}(\bR^n)}  \right) \le C(s,f),
\end{align*}
where we have used Lemma \ref{base} and H\"older's inequality. The left hand side of \eqref{p_weighted_calculation_part_1} can be written as
$$s(s+2)\int_{\bR^n} \frac{|u \cdot (x-y)|^2\eta_1 }{|x-y|^{s+4}} + s(s+2-n) \int_{\bR^n} \frac{\theta \eta_1 }{|x-y|^{s+2}} - \frac{s(s+4-n)}{2} \int_{\bR^n} \frac{|u|^2\eta_1 }{|x-y|^{s+2}}.$$
Since $s(s+2-n) < 0$ and $\frac{s(s+4-n)}{2} < 0,$
there exists a positive constant $C(s)$ depending only on $s$ and $n$, such that
\begin{equation}\label{p_weighted_calculation_part_2}
\frac{1}{C(s)} \left( \int_{\bR^n} \frac{|u \cdot (x-y)|^2\eta_1 }{|x-y|^{s+4}} + \int_{\bR^n} \frac{|u|^2\eta_1 }{|x-y|^{s+2}} \right) - \int_{\bR^n} \frac{\theta \eta_1 }{|x-y|^{s+2}} \le C(s,f).
\end{equation}
Replacing $\theta$ by $2\theta_+ - |\theta|$ in \eqref{p_weighted_calculation_part_2}, we have
\begin{align*}
&\int_{\bR^n} \frac{|u \cdot (x-y)|^2\eta_1 }{|x-y|^{s+4}} + \int_{\bR^n} \frac{|u|^2\eta_1 }{|x-y|^{s+2}} + \int_{\bR^n} \frac{|\theta| \eta_1 }{|x-y|^{s+2}}\\
 \le &C(s) \int_{\bR^n} \frac{\theta_+ \eta_1 }{|x-y|^{s+2}} + C(s,f)\\
 \le &C(s) \| \theta_+ \|_{L^{\frac{n}{2}}(B_{R_0 + 3})} \left( \int_{B_{R_0+3}} |x - y|^{-\frac{(s+2)n}{n-2}} \right)^{\frac{n-2}{n}} + C(s,f)\\
 \le &C(q,s,f) \left( \| \theta_+ \|_{L^{\frac{nq}{n-2q}}(B_{R_0 + 3})} + 1 \right),
\end{align*}
where we have used H\"older's inequality, the facts $\frac{(s+2)n}{n-2} < n$ and $\frac{nq}{n-2q} > \frac{n}{2}$. Since $p = \theta - \frac{|u|^2}{2}$, we have
\begin{align*}
\int_{\bR^n} \frac{|p|\eta_1}{|x-y|^{s+2}} &\le \int_{\bR^n} \frac{|\theta| \eta_1 }{|x-y|^{s+2}} + \frac{1}{2} \int_{\bR^n} \frac{|u|^2\eta_1 }{|x-y|^{s+2}}\\
&\le C(q,s,f) \left( \| \theta_+ \|_{L^{\frac{nq}{n-2q}}(B_{R_0 + 3})} + 1 \right).
\end{align*}
\end{proof}

\begin{lemma}\label{theta_u_relation_2_lemma}
For $f \in L^\infty(\bR^n)$ with $supp(f) \subset B_{R_0}$, let $u \in C_d^1(\bR^n)$ be a solution of \eqref{integral_equation} and $p$ be given by \eqref{pressure}. Then for any $\max \{2,n/4\} < q < \min\{4, n/2\}$, we have
\begin{equation}\label{theta_u_relation_2}
\|u\|_{L^q(B_{R_0})} \le C(q, R_0, f) \left(  \| \theta_+ \|^{\frac{q-2}{2q}}_{L^{\frac{nq}{n-2q}}(B_{R_0+3})} +  \| \theta_+ \|^{\frac{1}{2}}_{L^{\frac{nq}{n-2q}}(B_{R_0})} + 1 \right),
\end{equation}
where $C(q, R_0, f) > 0$ depends only on $n$, $q$, $R_0$, and an upper bound of $\|f\|_{L^\infty}$.
\end{lemma}

\begin{remark}
In the lemma above, in order for the interval $(\max \{2,n/4\}, \min\{4, n/2\})$ to be nonempty, we require $n/4 < 4$, which is equivalent to $n < 16$. This is the only place where the restriction of dimensions $n \le 15$ enters.
\end{remark}

\begin{proof}
We define a smooth cut-off function $\eta_2$ satisfying
$$\eta_2 = 1 ~ \mbox{in }B_{R_0}; \quad \eta_2 = 0 ~ \mbox{in }B_{R_0+1}^c; \quad |\nabla \eta_2| + |\nabla^2 \eta_2| \le C.
$$
For any $\max \{2,n/4\} < q < \min\{4, n/2\}$, we define
$$\varphi(y) := c_n \int_{\bR^n} \frac{1}{|x - y|^{n-2}} |p(x)|^{\frac{q-2}{2}} sgn(p(x)) \eta_2(x) \, dx, \quad y \in \bR^n,$$
where $c_n$ is the constant related to the fundamental solution of Laplace operator, such that
\begin{equation}\label{varphi_def}
-\Delta \varphi = |p|^{\frac{q-2}{2}} sgn(p) \eta_2,
\end{equation}
and $sgn$ is the sign function such that 
$$sgn(p) = \left\{
\begin{aligned}
1, \quad &\mbox{when }p>0;\\
0, \quad &\mbox{when }p=0;\\
-1, \quad &\mbox{when }p<0.
\end{aligned}
\right.$$
Fix an $s \in (0,n-4)$ such that
\begin{equation}\label{s_requirement}
\frac{2}{4-q} \cdot \left[ -n+2+ \frac{(s+2)(q-2)}{2} \right] > -n.
\end{equation}
By H\"older's inequality, we have
\begin{align*}
|\varphi(y)| &\le c_n \int_{\bR^n} \left( \frac{|p|\eta_2}{|x-y|^{s+2}} \right)^{\frac{q-2}{2}} |x-y|^{-n+2+ \frac{(s+2)(q-2)}{2}} \eta_2^{\frac{4-q}{2}} \, dx\\
&\le c_n \left( \int_{\bR^n} \frac{|p|\eta_2}{|x-y|^{s+2}} \right)^{\frac{q-2}{2}} \left( \int_{\bR^n} |x-y|^{\frac{2}{4-q} \cdot \left[ -n+2+ \frac{(s+2)(q-2)}{2} \right]  } \eta_2 \right)^{\frac{4-q}{2}}\\
&\le C\left( \int_{\bR^n} \frac{|p|\eta_2}{|x-y|^{s+2}} \right)^{\frac{q-2}{2}}
\end{align*}
due to \eqref{s_requirement}. Then by Lemma \ref{p_weighted_lemma}, we have
\begin{equation}\label{varphi_L_infty}
\|\varphi\|_{L^\infty(B_{R_0+1})} \le C(q,f) \left(  \| \theta_+ \|^{\frac{q-2}{2}}_{L^{\frac{nq}{n-2q}}(B_{R_0+3})} +  1 \right).
\end{equation}
We multiply \eqref{pressure_equation} by $\varphi \eta_2$ and integrate by parts, we have
\begin{align*}
- \int_{\bR^n} p \Delta \varphi \eta_2 + 2 \int_{\bR^n} \varphi \nabla p \cdot \nabla \eta_2 + \int_{\bR^n} p \varphi \Delta \eta_2  = \int_{\bR^n} \left( \partial_i u^j \partial_j u^i - \mbox{div }f \right) \varphi \eta_2,
\end{align*}
which implies, by \eqref{varphi_def}, H\"older's inequality, Lemma \ref{base} and \eqref{varphi_L_infty},
\begin{align}\label{p_q/2_power}
&\int_{\bR^n} |p|^{\frac{q}{2}} \eta_2^2 =  \int_{\bR^n} \left( \partial_i u^j \partial_j u^i - \mbox{div }f \right) \varphi \eta_2 - 2 \int_{\bR^n} \varphi \nabla p \cdot \nabla \eta_2 -  \int_{\bR^n} p \varphi \Delta \eta_2\nonumber \\
&\le C(f) \|\varphi\|_{L^\infty(B_{R_0+1})} \left( \| \nabla u\|_{L^2(\bR^n)}^2 + \| \nabla f\|_{L^\infty(\bR^n)} + \|\nabla p\|_{L^{\frac{n}{n-1}}(\bR^n)} + \| p\|_{L^{\frac{n}{n-2}}(\bR^n)} \right)\nonumber \\
&\le C(f) \|\varphi\|_{L^\infty(B_{R_0+1})} \le C(q,f) \left(  \| \theta_+ \|^{\frac{q-2}{2}}_{L^{\frac{nq}{n-2q}}(B_{R_0+3})} +  1 \right).
\end{align}
Since $|u|^2 \le 2|p| + 2\theta_+$, by \eqref{p_q/2_power} and H\"older's inequality, we have
\begin{align*}
\int_{B_{R_0}} |u|^q &\le C(q) \left( \int_{B_{R_0}} |p|^{\frac{q}{2}} + \int_{B_{R_0}} \theta_+^{\frac{q}{2}} \right)\\
&\le  C(q,f) \left(  \| \theta_+ \|^{\frac{q-2}{2}}_{L^{\frac{nq}{n-2q}}(B_{R_0+3})} +  1 + \| \theta_+ \|^{\frac{q}{2}}_{L^{\frac{nq}{n-2q}}(B_{R_0})}\right),
\end{align*}
which implies \eqref{theta_u_relation_2}.\\
\end{proof}

\begin{proof}[Proof of Proposition \ref{theta}]
We fix a $q$ satisfying $\max \{2,n/4\} < q < \min\{4, n/2\}$, then by Lemma \ref{theta_u_relation_1_lemma} and Lemma \ref{theta_u_relation_2_lemma}, we have
\begin{align*}
\| \theta_+ \|_{L^{\frac{nq}{n-2q}}(\bR^n)} &\le C(q,f) \left( \|u\|_{L^q(\{supp(f)\})} + 1 \right) \\
&\le C(q,f) \left(  \| \theta_+ \|^{\frac{q-2}{2q}}_{L^{\frac{nq}{n-2q}}(\bR^n)} +  \| \theta_+ \|^{\frac{1}{2}}_{L^{\frac{nq}{n-2q}}(\bR^n)} + 1 \right).
\end{align*}
This implies, using $0 < \frac{q-2}{2q} < 1$,
$$\| \theta_+ \|_{L^{\frac{nq}{n-2q}}(\bR^n)} \le C(q,f),$$
where $C(q,f) > 0$ depends only on $n,q,$ supp$(f)$, and $\|f\|_{L^\infty}$. Recall that $q > n/4$ is equivalent to $\frac{nq}{n-2q} > \frac{n}{2}$, Proposition \ref{theta} is proved with $r = \frac{nq}{n-2q}$.\\
\end{proof}

\section{Proof of Theorem \ref{decay_u_prop}}

This section is devoted to the proof of Theorem \ref{decay_u_prop}. First, we quantify the decay of $\nabla u$ in $L^2$ norm as follow.\\

\begin{lemma}\label{decay}
For $f \in L^\infty(\bR^n)$ with compact support, let $u \in C_d^1(\bR^n)$ be a solution of \eqref{integral_equation}. Then for any $\varepsilon > 0$, there exists an $R$ depending only on $\varepsilon$, $n$, an upper bound of the diameter of supp$(f)$, and an upper bound of $\|f\|_{L^\infty(\bR^n)}$, such that
$$\int_{\bR^n \setminus B_R} |\nabla u|^2 < \varepsilon.$$
\end{lemma}
\begin{proof}
In the following, $C$ denotes some positive constants depending
only on $n, R_0$ and an upper bound of $\|f\|_{L^\infty}$ whose values may change from line to line,  where $R_0 >0$ satisfies $supp(f) \subset B_{R_0}$.

Let $u \in C_d^1(\bR^n)$ be a solution of \eqref{integral_equation} and $p$ be given by \eqref{pressure}.  For all $i > R_0$, we take $\eta_i$ to be a smooth cut-off function such that 
$$\eta_i = 0 ~ \mbox{in }B_{i}; \quad \eta_i = 1 ~ \mbox{in }B_{i+1}^c; \quad |\nabla \eta_i|  \le C ~\mbox{in }E_i:= B_{i+1}\setminus B_i.
$$
Multiplying \eqref{NS} by $u \eta_i$ and integrating by parts, since $u$ and $p$ have the decay \eqref{u_decay} and \eqref{p_decay}, we have
\begin{equation}\label{grad_u_outside_ball}
\int_{\bR^n\setminus B_{i+1} }
|\nabla u|^2
\le C \int_{  E_i }
\left( |u||\nabla u|+ |u|^3 +|p||u|\right).
\end{equation}
For $m \ge 2l$, we have
\begin{align*}
\int_{\bR^n}  \left( |\nabla u|^2 +|u|^{ \frac {2n}{n-2} }
+|p|^{ \frac n{n-2} }\right)
&\ge  \sum_{ i=l} ^{l+m} \int_{ E_i }
 \left( |\nabla u|^2 +|u|^{ \frac {2n}{n-2} }
+|p|^{ \frac n{n-2} }\right)\\
&\ge  
\min_{ l\le i\le l+m} i  \int_{ E_i }
 \left( |\nabla u|^2 +|u|^{ \frac {2n}{n-2} }
+|p|^{ \frac n{n-2} }\right) \sum_{ i=l} ^{l+m} \frac 1i\\
&\ge  \frac 1C \left(\log \left(\frac ml\right)\right) \min_{ l\le i\le l+m} i  \int_{ E_i }
 \left( |\nabla u|^2 +|u|^{ \frac {2n}{n-2} }
+|p|^{ \frac n{n-2} }\right). 
\end{align*}
It follows that, by Lemma \ref{base}, for some $i\in \{l, l+1, \cdots, l+m\}$,
\begin{equation}
\int_{ E_i }
\left( |\nabla u|^2 +|u|^{ \frac {2n}{n-2} }
+|p|^{ \frac n{n-2} }\right)
\le \frac {C}{ i \log (\frac ml) }.
\label{1}
\end{equation}
By H\"older's inequality and \eqref{1}, we can estimate
\begin{align}\label{grad_u_part_1}
\int_{ E_i}|u| |\nabla u|
&\le \|u\|_{ L^{ \frac {2n}{n-2} } (E_i)  }
\|\nabla u\|_{ L^2  (E_i)  }
|E_i|^{ \frac 1n}\nonumber\\
&\le C i ^{ \frac {n-1}n } \|u\|_{ L^{ \frac {2n}{n-2} } (E_i)  }
\|\nabla u\|_{ L^2  (E_i)  }\le C \left(  \log \left(\frac ml \right)  \right)^{ \frac {1-n}n }.
\end{align}
We recall that by Corollary \ref{u_C1_coro}, we have
\begin{equation}\label{u_L_infty}
\| u \|_{L^\infty(\bR^n)} \le C.
\end{equation}
When $n \ge 6$, by  \eqref{u_L_infty}, \eqref{1}, and H\"older's inequality, we have
\begin{align}
\int_{ E_i}|u|^3
&\le \|u\|_{ L^\infty(E_i)} ^{ \frac {n-6}{n-2} } 
\|u\|_{ L^{ \frac {2n}{n-2} } (E_i)  }^{  \frac {2n}{n-2} }
\le C \left(  \log \left(\frac ml\right)  \right)^{ - 1 },\label{grad_u_part_2} \\
\int_{ E_i}
|p||u|&\le \|p\|_{ L^{ \frac n{n-2}} (E_i)  }
\|u\| _{ L^{ \frac n2} (E_i)  }
 \le 
\|u\|_{ L^\infty(E_i) } ^{ \frac {n-6}{n-2} }
\|p\|_{ L^{ \frac n{n-2}} (E_i)  } 
\|u\| _{ L^{ \frac {2n}{n-2} } (E_i)  } ^{ \frac {4}{n-2} } \nonumber\\ 
&\le C 
\|p\|_{ L^{ \frac n{n-2}} (E_i)  }
\|u\| _{ L^{ \frac {2n}{n-2} } (E_i)  } ^{ \frac {4}{n-2} }\le  C \left(  \log \left(\frac ml\right)  \right)^{ - 1 }.\label{grad_u_part_3}
\end{align}
When $n = 5$, by H\"older's inequality and \eqref{1},
\begin{align}
&\int_{E_i} |u|^3 \le |E_i|^{\frac{6-n}{2n}}\|u\|_{ L^{ \frac {2n}{n-2} } (E_i)  }^3 \le Ci^{\frac{(6-n)(n-1)}{2n}}\|u\|_{ L^{ \frac {2n}{n-2} } (E_i)  }^3 \le C\left(  \log \left(\frac ml\right)  \right)^{ - \frac{9}{10} },\label{grad_u_part_4} \\
&\int_{E_i} |p||u| \le |E_i|^{\frac{6-n}{2n}}\|p\|_{ L^{ \frac n{n-2}} (E_i)  } \|u\| _{ L^{ \frac {2n}{n-2} } (E_i)  }\le C\left(  \log \left(\frac ml\right)  \right)^{ - \frac{9}{10} }.\label{grad_u_part_5}
\end{align}
By \eqref{grad_u_outside_ball}, \eqref{grad_u_part_1}, \eqref{grad_u_part_2}, \eqref{grad_u_part_3}, \eqref{grad_u_part_4} and \eqref{grad_u_part_5}, we have
\begin{equation*}
\int_{\bR^n \setminus B_{i+1}} |\nabla u|^2 \le \int_{E_i}|u| |\nabla u| + |u|^3 + |p||u|\,dS \le C \left(  \log \left(\frac ml\right)  \right)^{ - \frac{9}{10} }.
\end{equation*}
Taking $m = l^2$, Lemma \ref{decay} is proved.
\end{proof}

Now we are ready to prove Theorem \ref{decay_u_prop}, with the help of Theorem \ref{regularity_criterion}, Corollary \ref{u_C1_coro}, and Lemma \ref{decay}.\\

\begin{proof}[Proof of Theorem \ref{decay_u_prop}]
Let $u \in C_d^1(\bR^n)$ be a solution of \eqref{integral_equation} and $p$ be given by \eqref{pressure}. By H\"older's inequality and \eqref{u_base}, for any $x_0 \in \bR^n$ and $R > 0$, we have
\begin{align}\label{v_L2}
R^{-(n-2)}\int_{B_R(x_0)} |u|^2 &\le R^{-(n-2)} \left( \int_{B_R(x_0)} |u|^{\frac{2n}{n-2}} \right)^{\frac{n-2}{n}} \left( \int_{B_R(x_0)} 1 \right)^{\frac{2}{n}}\nonumber\\
&\le CR^{-(n-4)} \left( \int_{B_R(x_0)} |u|^{\frac{2n}{n-2}} \right)^{\frac{n-2}{n}} \le C_1R^{-(n-4)},
\end{align}
where $C_1 > 0$ depends only on $n$, supp$(f)$, and $\|f\|_{L^\infty}$.
We choose $\varepsilon_0$ as in Theorem \ref{regularity_criterion} with $M_0 = C_1$.
Because of \eqref{u_C1}, for any $x_1 \in \bR^n$, we have
$$r^{-(n-4)}\int_{B_r(x_1)} |\partial_i u_j - \partial_j u_i|^2 \le C_2r^4,$$
where $C_2 >0$ depends only on $n$, supp$(f)$, and $\|f\|_{L^\infty}$. Therefore, one can choose $r_1$ such that $C_2r_1^4 < \varepsilon_0$, and hence
\begin{equation}\label{condition_part1}
r^{-(n-4)}\int_{B_r(x_1)} |\partial_i u_j - \partial_j u_i|^2 < \varepsilon_0, \quad \mbox{for all } 0 < r < r_1,\quad x_1 \in \bR^n.
\end{equation}
By Lemma \ref{decay}, there exists an $R_0$ depending on $\varepsilon_0$, $r_1$, $n$, supp$(f)$, and $\|f\|_{L^\infty}$, such that
$$\int_{\bR^n \setminus B_{R_0}} |\partial_i u_j - \partial_j u_i|^2 \le C \int_{\bR^n \setminus B_{R_0}} |\nabla u|^2 < \varepsilon_0 r_1^{n-4}.$$
For any $R > R_0$, $|x_0| = 3R$ and $x_1 \in B_R(x_0)$, we have
\begin{equation}\label{condition_part2}
r^{-(n-4)}\int_{B_r(x_1)} |\partial_i u_j - \partial_j u_i|^2 \le r^{-(n-4)} \int_{\bR^n \setminus B_{R_0}} |\partial_i u_j - \partial_j u_i|^2 < \varepsilon_0, ~\forall r_1 \le r < R/2. 
\end{equation}
Combining \eqref{condition_part1} and \eqref{condition_part2}, we have
\begin{equation}\label{condition_v}
r^{-(n-4)}\int_{B_r(x_1)} |\partial_i u_j - \partial_j u_i|^2 < \varepsilon_0, \quad \forall r<R/2, ~ x_1 \in B_R(x_0).
\end{equation}

In the following, unless stated otherwise, $C$ denotes some positive constants depending
only on $n$, an upper bound of the diameter of supp$(f)$, and an upper bound of $\|f\|_{L^\infty}$ whose values may change from line to line.

We set $v(x) = Ru(Rx+x_0)$, then we have, by \eqref{v_L2} and \eqref{condition_v},
$$\|v\|_{L^2(B_1)} \le CR^{-(n-4)/2},$$
$$r^{-(n-4)}\int_{B_r(x_1)} |\partial_i v_j - \partial_j v_i|^2 < \varepsilon_0, \quad \forall r<1/2, ~ x_1 \in B_1,$$
and $v$ satisfies the equation
$$\left\{
\begin{aligned}
- \Delta v  + \nabla \pi &= -(v \cdot \nabla) v,\\
\mbox{div}~v&=0
\end{aligned}
\right.
\quad \mbox{in }B_1,$$
where $\pi(x) = R^2p(Rx+x_0)$. Applying Theorem \ref{regularity_criterion} on $v$ gives us 
$$\| \nabla v\|_{L^\infty(B_{1/2})} \le C,$$ 
and hence
$$\| (v \cdot \nabla) v \|_{L^2(B_{1/2})} \le CR^{-(n-4)/2}.$$
By the interior estimate of the stationary Stokes equations (see, e.g., \cite[Theorem 2.2]{ST}), we  have
$$\|v\|_{W^{2,2}(B_{1/4})} \le C(\| (v \cdot \nabla) v \|_{L^2(B_{1/2})} + \|v\|_{L^2(B_{1/2})}) \le CR^{-(n-4)/2},$$
which implies 
$$\| v \|_{L^{\frac{2n}{n-4}}(B_{1/4})}\le CR^{-(n-4)/2}$$
by Sobolev inequality. Then we have
$$\| (v \cdot \nabla) v \|_{L^{\frac{2n}{n-4}}} \le CR^{-(n-4)/2},$$
and we can repeat the process above. For any $q < \infty$, after repeating this process finite times, we have 
$$\|v\|_{W^{2,q}(B_{1/8})} \le C(q)R^{-(n-4)/2},$$
which implies
$$\|v\|_{C^1(\overline{B}_{1/8})}\le CR^{-(n-4)/2}.$$
Reversing the change of variable will give
$$|u(x_0)| \le C|x_0|^{-n/2 + 1}, \quad \mbox{and} \quad |\nabla u(x_0)|\le C|x_0|^{-n/2}.$$
Because of \eqref{u_C1}, so far we have shown that
\begin{equation}\label{u_first_decay}
|u(x)| \le \frac{C}{(1+ |x|)^{n/2 - 1}}, \quad \mbox{and } \quad |\nabla u(x)| \le \frac{C}{(1+ |x|)^{n/2}}.
\end{equation}
Now we use \eqref{integral_equation}, the integral equation $u$ satisfies, by \eqref{u_first_decay} and Lemma \ref{calculus_lemma},
\begin{align*}
|u(x)| &\le \int_{\bR^n} |U(x-y)| \left( |f(y)| + |u(y)| |\nabla u(y)|\right) \, dy\\
&\le C \int_{\bR^n} \frac{1}{|x-y|^{n-2}} \frac{1}{(1 + |y|)^{n-1}} \, dy \le \frac{C}{(1+|x|)^{n-3}},\\
|\nabla u(x)| &\le \int_{\bR^n} |\nabla U(x-y)| \left( |f(y)| + |u(y)| |\nabla u(y)|\right) \, dy\\
&\le C \int_{\bR^n} \frac{1}{|x-y|^{n-1}} \frac{1}{(1 + |y|)^{n-1}} \, dy \le \frac{C}{(1+|x|)^{n-2}}.
\end{align*}
Hence Theorem \ref{decay_u_prop} is proved.\\
\end{proof}

\section{Proof of Theorem \ref{main}}

From now on we fix an arbitrary external force $f \in L^\infty(\bR^n)$ with compact support. For $v \in C_d^1(\bR^n)$ and $t \in [0,1]$, we consider the vector-valued function $u = (u_1, \cdots , u_n)$ given by
\begin{equation}\label{integral_equation_2}
u_i(x) = \int_{\bR^n} U_{ij}(x-y) \left(tf_j(y) - v_k(y) \partial_k v_j(y)\right) \, dy.
\end{equation}
We define an operator
\begin{align*}
F: [0,1] \times C_d^1(\bR^n) &\to C_d^1(\bR^n),\\
(t,v) &\mapsto u,
\end{align*}
where $u$ is given by \eqref{integral_equation_2}. By Lemma \ref{calculus_lemma}, we have
\begin{align}
|F(t,v)(x)| &\le \int_{\bR^n} |U(x-y)| \left( |f(y)| + |v(y)| |\nabla v(y)|\right) \, dy \nonumber\\
&\le C \int_{\bR^n} \frac{1}{|x-y|^{n-2}} \frac{1}{(1 + |y|)^{2n-5}} \, dy \le \frac{C\log (2 + |x|)}{(1+|x|)^{n-2}},\label{F_decay}\\
|\nabla F(t,v)(x)| &\le \int_{\bR^n} |\nabla U(x-y)| \left( |f(y)| + |u(y)| |\nabla u(y)|\right) \, dy\nonumber\\
&\le C \int_{\bR^n} \frac{1}{|x-y|^{n-1}} \frac{1}{(1 + |y|)^{2n-5}} \, dy \le \frac{C\log (2 + |x|)}{(1+|x|)^{n-1}}, \label{grad_F_decay}
\end{align}
where $C > 0$ depends on $n$, an upper bound of the diameter of supp$(f)$, and upper bounds of $\|f\|_{L^\infty}$ and $\|v\|_{C_d^1}$. Therefore $F$ is well-defined. 

A fixed point of $F(1,\cdot)$ in $C_d^1(\bR^n)$ is a solution $u \in C_d^1(\bR^n)$ to the integral equation \eqref{integral_equation}. We will show the existence of such a fixed point by using the Leray Schauder degree theory. First, we show that the operator $F$ is compact.\\

\begin{lemma}
$F: [0,1] \times C_d^1(\bR^n)  \to C_d^1(\bR^n)$ is compact.
\end{lemma}
\begin{proof}
Let $\{(t^i, v^i)\}$ be a bounded sequence in $[0,1] \times C_d^1(\bR^n)$, we will show that there exists a $\xi \in C_d^1(\bR^n)$, and a subsequence, still denoted by $\{(t^i, v^i)\}$, such that $F(t^i,v^i) \to \xi$ in $C_d^1(\bR^n)$.

First we will show that, after passing to a subsequence, there exists a $\xi \in C^1(\bR^n)$ such that
$$F(t^i, v^i) \to \xi \quad \mbox{in}~C^1_{loc}(\bR^n).$$
It suffices to show that
\begin{equation}\label{F_W_2q}
\| F(t^i, v^i) \|_{W^{2,q}(B_R)} \le C(q,R), \quad \forall R >1,~ \forall 1 < q < \infty,
\end{equation}
where $C(q,R) > 0$ depends only on $n,q,R$, but does not depend on $i$. For any $R > 1$, and for any $x \in B_R$, we can write
\begin{align*}
F(t^i, v^i)(x) &= \left( \int_{\{|y| < 2R\}} + \int_{\{|y| \ge 2R\}} \right) U(x-y)\left( t^if(y) - (v^i(y) \cdot \nabla) v^i(y) \right) \, dy\\
&=: I(x) + II(x).
\end{align*}
By the Calderon-Zygmund estimate, we have
$$\|I\|_{W^{2,q}(B_R)} \le C \left( \| f \|_{L^q(B_{2R})} + \| (v \cdot \nabla) v \|_{L^q(B_{2R})} \right) \le C(q,R).$$
For $l = 0,1,2$,
\begin{align*}
\left| \nabla^l II(x) \right| &\le \int_{\{|y| \ge 2R\}} \left| \nabla^l U(x - y) \right| \left( |f(y)| + |v(y)| \left| \nabla v(y) \right| \right)\, dy\\
&\le C \int_{\{|y| \ge 2R\}} \frac{1}{|y|^{n-2+l}} \frac{1}{(1+|y|)^{2n-5}} \,dy \le C(R), \quad \forall |x| < R.
\end{align*}
Therefore, \eqref{F_W_2q} follows.

For any $\varepsilon > 0$, by \eqref{F_decay} and \eqref{grad_F_decay}, there exists an $R > 1$ depending only on $\varepsilon$ and $n$, such that
$$|F(t^i, v^i)(x)|(1+|x|)^{n-3} < \varepsilon, \quad |\nabla F(t^i, v^i)(x)|(1+|x|)^{n-2} < \varepsilon, \quad \forall |x| > R.$$
Therefore $\xi \in C_d^1(\bR^n)$ and, after passing to a subsequence, $F(t^i, v^i) \to \xi$ in $C_d^1(\bR^n)$.\\
\end{proof}

\begin{proof}[Proof of Theorem \ref{main}]

Fix any $f \in L^\infty(\bR^n)$ with compact support. Showing the existence of a solution in $C_d^1(\bR^n)$ to \eqref{NS} is equivalent to showing the existence of a solution of 
$$u - F(1,u)= 0.$$ 
By Proposition \ref{decay_u_prop}, we know that there exists a constant $M$ such that
$$\|u\|_{C_d^1(\bR^n)} \le M,$$
for any solution $u \in C_d^1(\bR^n)$ of $u - F(t,u) = 0$, for any $t \in [0,1]$. So $u - F(t, u) = 0$ has no solution on $\partial B_{2M}$, where $B_{2M}:= \{u \in C_d^1(\bR^n); \|u\|_{C_d^1(\bR^n)} < 2M\}$. The Leray-Schauder degree 
$$deg(Id - F(t,\cdot),B_{2M},0)$$
is well defined for $t \in [0,1]$, and, by the homotopy invariance, it is independent of $t$. In particular,
$$deg(Id - F(1,\cdot),B_{2M},0) = deg(Id - F(0,\cdot),B_{2M},0).$$
See, e.g., Section 2.3 in \cite{Nir}. 

$u - F(0,u) = 0$ is equivalent to
\begin{equation*}
\left\{
\begin{aligned}
- \Delta u  + \nabla p &=  - (u \cdot \nabla) u\\
\mbox{div }u &= 0
\end{aligned}
\right.
\quad \mbox{in } \bR^n.
\end{equation*}
Therefore, $u \equiv 0$ is the only solution in $C_d^1(\bR^n)$ to the equation $u - F(0,u) = 0$. Since $F_u(0,0) = 0$, we have (see, e.g., \cite[Theorem 2.8.1]{Nir})
$$deg(Id - F(0,\cdot),B_{2M},0) = 1 \neq 0.$$
This implies the existence of $u \in C_d^1(\bR^n)$ that satisfies the integral equation \eqref{integral_equation}. Let $p$ be given by \eqref{pressure}, then $(u,p)$ is a regular solution of \eqref{NS}. Since the solution $u$ we obtain satisfies the bound
$$\|u\|_{C_d^1(\bR^n)} \le 2M.$$
It follows, from the calculations in \eqref{F_decay} and \eqref{grad_F_decay},
$$|u(x)| \le \frac{C(\varepsilon)}{(1+|x|)^{n-2-\varepsilon}} \quad \mbox{and} \quad |\nabla u(x)| \le \frac{C(\varepsilon)}{(1+|x|)^{n-1-\varepsilon}},$$
for some $\varepsilon > 0$.
Estimate \eqref{reasonable_decay} follows from Lemma \ref{calculus_lemma} and similar calculations as above.\\
\end{proof}

\section*{Acknowledgment}
The first named author thanks Vladim\'{\i}r \v{S}ver\'{a}k for bringing to his attention in 2012 the work of Frehse and Rů\v{z}i\v{c}ka \cite{FR0}.

\setcounter{section}{0}
\setcounter{theorem}{0}
\setcounter{equation}{0}
\renewcommand{\theequation}{\thesection.\arabic{equation}}
\setcounter{figure}{0}
\setcounter{table}{0}
 
\appendix
\section*{Appendix}
\renewcommand{\thesection}{A} 
 
\subsection{Another proof of Proposition \ref{aprioriestimate}}

In this section, we provide another proof of Proposition \ref{aprioriestimate} using the arguments in \cite{FR1} and \cite{FR6}. First, let us recall the definition of Sobolev-Morrey spaces and state an embedding theorem in \cite{A}.

\begin{definition}
Let $\Omega$ be a bounded smooth domain, $1 \le p < \infty, 0 \le \lambda < n$. We say a function $f \in L^{p,\lambda}(\Omega)$, if
$$\sup_{x \in \Omega, r > 0} \frac{1}{r^\lambda} \int_{B_r(x) \cap \Omega} |f|^p < \infty,$$
with norm
$$\|f\|_{L^{p,\lambda}(\Omega)}= \left( \sup_{x \in \Omega, r > 0} \frac{1}{r^\lambda} \int_{B_r(x) \cap \Omega} |f|^p \right)^{\frac{1}{p}}.$$
We say a function $g \in W^{k,p,\lambda}(\Omega)$, if $\nabla^\alpha g \in L^{p,\lambda}(\Omega)$, for all $|\alpha| \le k$, with norm
$$\|g\|_{W^{k,p,\lambda}(\Omega)} = \sum_{|\alpha| \le k} \|\nabla^\alpha g\|_{L^{p,\lambda}(\Omega)}.\\$$
\end{definition}

\begin{theorem}\label{SM_embedding_thm}
Let $1 < p < \infty, 0 \le \lambda < n$. If $f \in W^{1,p,\lambda}(B_1)$, then $f \in  L^{p^*,\lambda}(B_1)$, where
$$\left\{\begin{aligned}
\frac{1}{p^*} = \frac{1}{p} - \frac{1}{n - \lambda},\quad & \mbox{if } p < n - \lambda;\\
p^* \mbox{can be any finite number},\quad & \mbox{if } p \ge n - \lambda.
\end{aligned}\right.$$
Furthermore,
\begin{equation}\label{SM_embedding}
\|f\|_{L^{p^*,\lambda}(B_1)} \le C(\|\nabla f\|_{L^{p,\lambda}(B_1)} + \|f\|_{L^{1}(B_1)}),
\end{equation}
where $C>0$ depends only on $n, p$ and $\lambda$.\\
\end{theorem}

\begin{proof}
For $x,y \in B_1$, we have
\begin{align*}
|f(x) - f(y)| = \left| \int_0^1 \frac{d}{dt}f(ty +(1-t)x) \,dt \right| \le 2\int_0^1 |\nabla f(ty +(1-t)x)| \,dt.
\end{align*}
It follows that
\begin{align*}
|f(x) - f_{B_1}| &\le \frac{1}{|B_1|} \int_{B_1} |f(x) - f(y)| \, dy\\
&\le C \int_{B_1} \int_0^1 |\nabla f(ty +(1-t)x)| \,dt dy\\
&\le C \int_{B_1} \frac{|\nabla f(z)|}{|x-z|^{n-1}} \, dz,
\end{align*}
where $C > 0$ depends only on $n$, $f_{B_1}$ denotes the average of $f$ over $B_1$. Therefore,
\begin{equation}\label{f_representation}
|f(x)| \le C \int_{B_1} \frac{|\nabla f(z)|}{|x-z|^{n-1}} \, dz + |f_{B_1}|.
\end{equation}
Estimate \eqref{SM_embedding} follows after applying a Morrey estimate on Riesz potentials in \cite[Theorem 3.2]{A} to \eqref{f_representation}. 
\end{proof}

We also need the following Sobolev-Morrey space analogue of the interior Sobolev space estimates for the stationary Stokes equations proved in \cite{ST}. This can be proved by using the Morrey space estimates instead of the $L^p$ estimates on Calderon-Zygmund operators in the arguments there.

\begin{theorem}\label{Morrey_Stokes_thm}
Let $(u,p)$ be a smooth solution of the stationary Stokes equations
$$\left\{
\begin{aligned}
- \Delta u + \nabla p &= f,\\
\mbox{div }u &= 0,
\end{aligned}
\right.
\quad \mbox{in}~B_2 \subset \bR^n, 
$$
with smooth $f$. Then for $1 < q < \infty, 0 \le \lambda < n$,
\begin{equation*}
\| \nabla^2 u\|_{L^{q,\lambda}(B_1)} + \| \nabla p \|_{L^{q,\lambda}(B_1)} \le C( \|f\|_{L^{q,\lambda}(B_2)} + \|u\|_{L^1(B_2 \setminus B_1)}),
\end{equation*}
where $C > 0$ depends only on $n$, $q$ and $\lambda$.\\
\end{theorem}

\begin{proof}
We define
$$\tilde{u}_i:= U_{ij} \ast (f_j\chi_{B_2}), \quad \tilde{p} = P_j \ast (f_j\chi_{B_2}),$$
where $(U,P)$ is the fundamental solution of the stationary Stokes equations as \eqref{Stokes_fundamental_UP}, $\chi_{B_2}$ is the characteristic function on $B_2$. Then by the Morrey space estimates for the Calderon-Zygmund operators (see, e.g., \cite[Theorem 3]{CF}), we have
$$\| \nabla^2 \tilde{u} \|_{L^{p,\lambda}(\bR^n)} + \| \nabla \tilde{p} \|_{L^{p,\lambda}(\bR^n)} \le C\|f\|_{L^{p,\lambda}(B_2)}.$$
Let $v := u - \tilde{u}, \pi = p - \tilde{p}$, then $(v,\pi)$ satisfies
$$\left\{
\begin{aligned}
- \Delta v + \nabla \pi &= 0,\\
\mbox{div }v &= 0,
\end{aligned}
\right. \quad \mbox{in } B_2.$$
Therefore, $v$ is biharmonic. Indeed,
$$\partial_{iikk}v^j = \partial_{iij} \pi = \partial_{iji} \pi = \partial_{ijkk} v^i = 0.$$
By the interior estimates for biharmonic function (see, e.g., \cite{Bro}),
$$\| \nabla^2 v \|_{L^{\infty} (B_{1})} \le C \| v\|_{L^1(B_2 \setminus B_1)}.$$
For any $x_0 \in B_1, r <1/2,$
$$\frac{1}{r^\lambda} \int_{B_r(x_0) \cap B_1} |\nabla^2 v|^p \, dx \le Cr^{n-\lambda}\| \nabla^2 v \|_{L^{\infty} (B_{1})}^p \le C \|v\|_{L^1(B_2 \setminus B_1)}^p,$$
and hence
\begin{align*}
\frac{1}{r^\lambda} \int_{B_r(x_0)\cap B_1} |\nabla^2 u|^p \, dx &\le \frac{1}{r^\lambda} \int_{B_r(x_0)\cap B_1} |\nabla^2 v|^p \, dx  + \frac{1}{r^\lambda} \int_{B_r(x_0)\cap B_1} |\nabla^2 \tilde{u}|^p \, dx \\
&\le C (\|v\|_{L^1(B_2 \setminus B_1)}^p + \|f\|_{L^{p,\lambda}(B_2)}^p)\\
&\le C (\|u\|_{L^1(B_2 \setminus B_1)}^p + \|\tilde{u}\|_{L^1(B_2 \setminus B_1)}^p + \|f\|_{L^{p,\lambda}(B_2)}^p)\\
&\le C (\|u\|_{L^1(B_2 \setminus B_1)}^p + \|f\|_{L^{p,\lambda}(B_2)}^p).
\end{align*}
This gives the desired estimate of $\nabla^2 u$ on $B_1$. For the pressure part, we know
\begin{align*}
\|\nabla \pi \|_{L^{p,\lambda}(B_1)} = \|\Delta v\|_{L^{p,\lambda}(B_1)} &\le C\|\nabla^2 v\|_{L^{p,\lambda}(B_1)}\\
&\le C (\|u\|_{L^1(B_2 \setminus B_1)} + \|f\|_{L^{p,\lambda}(B_2)}).
\end{align*}
Therefore,
\begin{align*}
\|\nabla p \|_{L^{p,\lambda}(B_1)} &\le \|\nabla \pi \|_{L^{p,\lambda}(B_1)} + \|\nabla \tilde{p} \|_{L^{p,\lambda}(B_1)} \le C (\|u\|_{L^1(B_2 \setminus B_1)} + \|f\|_{L^{p,\lambda}(B_2)}).
\end{align*}
\end{proof}

\begin{proof}[Proof of Proposition \ref{aprioriestimate}]
We know from \eqref{u_morrey} that there exist positive constants $\beta$ and $C$ depending only on $n$, $C_0$, and a positive lower bound of $r - n/2$, such that
\begin{equation}\label{grad_u_morrey}
\| \nabla u\|_{L^{2,n-4+\beta}(B_{1/4})} \le C.
\end{equation}

If $\beta \in [2,4)$, we have, by Theorem \ref{SM_embedding_thm}, $\|u\|_{L^q(B_2)} \le C(q)$, for any $q < \infty$. Then Proposition \ref{aprioriestimate} follows from standard estimates for Stokes equations. Therefore, we only need to treat the case $\beta \in (0,2)$.

Rewrite the stationary Navier-Stokes equations \eqref{NS_B1} as
$$- \Delta u  + \nabla p = f - (u \cdot \nabla) u.$$
By Theorem \ref{SM_embedding_thm} and \eqref{grad_u_morrey}, 
$$\|u\|_{L^{s,n-4+ \beta}(B_{1/4})} \le C(\|\nabla u\|_{L^{2,n-4+ \beta}(B_{1/4})} +  \|u\|_{L^{1}(B_{1/4})}) \le C,$$
where $\frac{1}{s} = \frac{1}{2} - \frac{1}{4 - \beta}$. Thus, by Holder's inequality,
\begin{align*}
\| f - (u \cdot \nabla) u \|_{L^{r,n-4+ \beta}(B_{1/4})} &\le \| f \|_{L^{r,n-4+ \beta}(B_{1/4})}+ \|u\|_{L^{s,n-4+ \beta}(B_{1/4})}\|\nabla u\|_{L^{2,n-4+ \beta}(B_{1/4})}\\
&\le C,
\end{align*}
where $\frac{1}{r} =  \frac{1}{s} + \frac{1}{2}$. Then, by Theorem \ref{SM_embedding_thm} and Theorem \ref{Morrey_Stokes_thm}, we have
\begin{align*}
\| \nabla u\|_{L^{t,n-4+ \beta}(B_{1/8})} &\le C( \| \nabla^2 u\|_{L^{r,n-4+ \beta}(B_{1/8})} + \| \nabla u\|_{L^{1}(B_{1/8})} )\\
&\le C(\| f - (u \cdot \nabla) u \|_{L^{r,n-4+ \beta}(B_{1/4})} + \|u\|_{L^1(B_{1/4})})\le C,
\end{align*}
where $\frac{1}{t} = \frac{1}{r} - \frac{1}{4-\beta} = 1 - \frac{2}{4 - \beta}$. One can see that by this process, the regularity of $\nabla u$ has been improved from $L^{t_i, n-4+\beta}$ to $L^{t_{i+1}, n - 4 + \beta}$, where $\frac{1}{t_{i+1}} = \frac{1}{t_i} + \frac{1}{2} - \frac{2}{4-\beta}$. We can repeat this process final times to obtain
$$\| \nabla u \|_{L^{p, n - 4 + \beta}(B_{1/16})} \le C,$$
for some $p \ge n - 4 + \beta$. This implies, by Theorem \ref{SM_embedding_thm}, for any $q<\infty$,
$$\|u\|_{L^q(B_{1/16})} \le C(q).$$
Then we have, by standard estimates for Stokes equations,
$$|u(0)|+ |\nabla u(0)| \le C,$$
where $C>0$ depends only on $n$, $C_0$, and a positive lower bound of $r - n/2$. Since the problem is translation invariant, estimate \eqref{apriori} follows.\\
\end{proof}

\section*{Declarations}
\subsection*{Funding}
Yanyan Li was partially supported by NSF Grants DMS-1501004, DMS-2000261, and Simons Fellows Award 677077. Zhuolun Yang was partially supported by NSF Grants DMS-1501004 and DMS-2000261.
\subsection*{Conflicts of interests}
The authors have no relevant financial or non-financial interests to disclose.
\subsection*{Availability of data and material}
Not applicable.
\subsection*{Code availability}
Not applicable.

\providecommand{\bysame}{\leavevmode\hbox to3em{\hrulefill}\thinspace}

\end{document}